\documentclass[14pt,a4paper]{amsart}
\usepackage{hyperref}
\usepackage{geometry}
\usepackage{amsmath,amssymb,amsthm,mathrsfs,amsfonts,dsfont}
\usepackage{tikz}
\usetikzlibrary{shapes}
\usetikzlibrary{plotmarks}
\usetikzlibrary{arrows}
\usetikzlibrary{positioning}
\usepackage{tkz-graph}
\theoremstyle{definition}
\newtheorem{defn}{Definition}[section]

\newtheorem{thm}[defn]{Theorem}

\newtheorem{lem}[defn]{Lemma}
\newtheorem{question}[defn]{Question}

\newtheorem{obs}[defn]{Observation}
\newtheorem{claim}[defn]{claim}

\title[Product of graphs, graph homomorphism, Antichains, cofinal subsets of posets]{Chromatic number of the product of graphs, graph homomorphisms, Antichains and cofinal subsets of posets without AC}
\author{Amitayu Banerjee}
\author{Zal\'{a}n Gyenis}
\address{Department of Logic, Institute of Philosophy, E\"otv\"os Lor\'and University,
M\'{u}zeum krt. 4/i Budapest, H-1088 Hungary}
\email{banerjee.amitayu@gmail.com}
\address{Institute of Philosophy, Department of Logic, Jagiellonian University, Grodzka 52, 33-332, Krak\'{o}w, Poland}
\email{zalan.gyenis@gmail.com}


\begin{document}
\maketitle
\begin{abstract}In set theory without the Axiom of Choice (AC), we observe new relations of the following statements with weak choice principles.
\begin{itemize}
    \item If in a partially ordered set, all chains are finite and all antichains are countable, then the set is countable.
    \item If in a partially ordered set, all chains are finite and all antichains have size $\aleph_{\alpha}$, then the set has size $\aleph_{\alpha}$ for any regular $\aleph_{\alpha}$.
    \item CS (Every partially ordered set without a maximal element has two disjoint cofinal subsets).
    \item  CWF (Every partially ordered set has a cofinal well-founded subset).
    \item DT (Dilworth's decomposition theorem for infinite p.o.sets of finite width).
\end{itemize}
We also study a graph homomorphism problem and a problem due to Andr\'{a}s Hajnal without AC.
Further, we study a few statements restricted to linearly-ordered structures without AC.
\end{abstract}

\section{Introduction}
Firstly, the first author observes the following in ZFA (Zermelo-Fraenkel set theory with atoms). 

\begin{enumerate}
    \item In \textbf{Problem 15, Chapter 11 of \cite{KT2006}}, applying Zorn's lemma, Komj\'{a}th and Totik proved the statement ``Every partially ordered set without a maximal element has two disjoint cofinal subsets'' (CS). In \textbf{Theorem 3.26 of \cite{THS2016}}, Tachtsis, Howard and Saveliev proved that CS does not imply `there are no amorphous sets' in ZFA. We observe that CS $\not\rightarrow AC^{\omega}_{fin}$ (the axiom of choice for countably infinite familes of non-empty finite sets), CS $\not\rightarrow AC^{-}_{n}$ (`Every infinite family of $n$-element sets has a partial choice function') \footnote{It is easy to see that $AC^{-}_{n}$ follows from $AC_{n}$ (Axiom of choice for $n$-element sets).} for every $2\leq n < \omega$ and CS $\not\rightarrow LOKW_{4}^{-}$ (Every infinite linearly orderable family $\mathcal{A}$ of $4$-element sets has a partial Kinna--Wegner selection function)\footnote{We denote the principle {`Every infinite linearly orderable family $\mathcal{A}$ of $n$-element sets has a partial Kinna--Wegner selection function'} by $LOKW^{-}_{n}$ (see \cite{HT2019}).} in ZFA.
    
    \item  In \textbf{Problem 14, Chapter 11 of \cite{KT2006}}, applying the well-ordering theorem, Komj\'{a}th and Totik proved the statement ``Every partially ordered set has a cofinal well-founded subset'' (CWF). In \textbf{Theorem 10(ii) of \cite{Tac2017}}, Tachtsis proved that CWF holds in the basic Fraenkel model. Moreover, in \textbf{Lemma 5 of \cite{Tac2017}}, Tachtsis proved that CWF is equivalent to AC in ZF. We observe that CWF $\not\rightarrow AC^{\omega}_{fin}$, CWF $\not\rightarrow AC^{-}_{n}$  for every $2\leq n < \omega$ and CWF $\not\rightarrow LOKW_{4}^{-}$ in ZFA.
    
    \item In \textbf{Problem 7, Chapter 11 of \cite{KT2006}}, applying Zorn's lemma, Komj\'{a}th and Totik proved that if in a partially ordered set, all chains are finite and all antichains are countable, then the set is countable. We observe that {\em `If in a partially ordered set, all chains are finite and all antichains are countable, then the set is countable'} $\not\rightarrow AC_{n}^{-} (\forall n\geq 2)$, 
    $\not\rightarrow $ `There are no amorphous sets' in ZFA which are new results. Moreover, we prove that
    {\em `For any regular $\aleph_{\alpha}$, if in a partially ordered set, all chains are finite and all antichains have size $\aleph_{\alpha}$, then the set has size $\aleph_{\alpha}$'} $\not\rightarrow AC_{n}^{-} (\forall n\geq 2)$, $\not\rightarrow $ `There are no amorphous sets' in ZFA. 

    \item Dilworth \cite{Dil1950} proved the following statement: {\em `If $\mathbb{P}$ is an arbitrary p.o.set, and $k$ is a natural number such that $\mathbb{P}$ has no antichains of size $k + 1$ while at least one $k$-element subset of $\mathbb{P}$ is an antichain, then $\mathbb{P}$ can be partitioned into k chains'}, we abbreviate by DT (see \textbf{Problem 4, Chapter 11 of \cite{KT2006}} also). Tachtsis \cite{Tac2019} investigated the possible placement of DT in the hierarchy of weak choice principles. He proved that DT does not imply $AC^{\omega}_{fin}$ as well as $AC_{2}$ (Every family of pairs has a choice function). We observe that DT does not imply $AC_{n}^{-}$ for any $2\leq n < \omega$ in ZFA. In particular, we observe that DT holds in the permutation model of \textbf{Theorem 8} of \cite{HT2019}, due to Halbeisen and Tachtsis.  We also observe that a weaker form of \L{}o\'{s}'s lemma (Form 253 of \cite{HR1998}) fails in the permutation model of \textbf{Theorem 8} of \cite{HT2019}.
    
    \item In \textbf{Theorem 4.5.2} of \cite{Kom}, Komj\'{a}th sketched the following generalization of the $n$-coloring theorem (For every graph $G=(V,E)$ if every finite subgraph of $G$ is $n$-colorable then $G$ is $n$-colorable) applying the Boolean prime ideal theorem (BPI): {\em `For an infinite graph $G=(V_{G}, E_{G})$ and a finite graph $H=(V_{H}, E_{H})$, if every finite subgraph of $G$ has a homomorphism into $H$, then so has $G$'} we abbreviate by $\mathcal{P}_{G,H}$. We observe that if $X\in \{AC_{3}, AC^{\omega}_{fin}\}$, then $\mathcal{P}_{G,H}$ restricted to finite graph $H$ with 2 vertices does not imply $X$ in ZFA.
\end{enumerate}  

Secondly, we study a weaker formulation of a problem due to Andr\'{a}s Hajnal in ZFA.

\begin{enumerate}
    \item In \textbf{Theorem 2} of \cite{Haj1985}, Hajnal proved that if the chromatic number of a graph $G_{1}$ is finite (say $k<\omega$), and the chromatic number of another graph $G_{2}$ is infinite, then the chromatic number of $G_{1}\times G_{2}$ is $k$ using the G\"{o}del's Compactness theorem. In the solution of \textbf{Problem 12, Chapter 23 of \cite{KT2006}}, Komj\'{a}th provided another argument using the Ultrafilter lemma. For a natural number $k<\omega$, we denote  by $\mathcal{P}_{k}$ the following statement. 
    \begin{center}
    {\em `$\chi(E_{G_{1}})=k<\omega$ and $\chi(E_{G_{2}})\geq\omega$ implies $\chi(E_{G_{1}\times G_{2}})=k$.'}
    \end{center}
    We observe that if $X\in \{AC_{3}, AC^{\omega}_{fin}\}$, then  $\mathcal{P}_{k}\not\rightarrow X$ in ZFA when $k=3$. 
\end{enumerate}
Lastly, we study a few algebraic and graph-theoretic statements restricted to linearly-ordered structures without AC. We abbreviate the statement `The union of a well-orderable family of finite sets is well-orderable' by  $UT(WO,fin,WO)$. In \textbf{Theorem 3.1 (i)} of \cite{Tac2019}, Tachtsis proved DT for well-ordered infinite p.o.sets with finite width in ZF applying the following theorem.

\begin{thm}{\em (\textbf{Theorem 1} of \cite{Loeb1965}).}
{\em Let $\{X_{i}\}_ {i \in I}$ be a family of compact spaces which is indexed by a set $I$ on which there is a well-ordering $\leq$. If $I$ is an infinite set and there is a choice function $F$ on the collection $\{C$ : C is
closed, $C \not= \emptyset, C \subset X_{i}$ for some $i\in I\}$, then the product space
$\Pi_{i\in I}X_{i}$ is compact in the product topology.}
\end{thm}

Using the same technique from \textbf{Theorem 3.1} of \cite{Tac2019}, we prove a few algebraic and graph-theoretic statements restricted to well-ordered sets, either in ZF or in ZF + $UT(WO,fin,WO)$. Consequently, those statements restricted to linearly ordered sets are true, in permutation models where LW (Every linearly ordered set can be well-ordered) holds. In particular, we observe the following.

\begin{enumerate}
    \item In \textbf{Theorem 18 of \cite{HT2013}}, Howard and Tachtsis obtained that for every finite field $\mathcal{F}=\langle F,...\rangle$, for every nontrivial vector space $V$ over $\mathcal{F}$, there exists a non-zero linear functional $f:V\rightarrow F$ applying BPI. Fix an arbitrary $2\leq n < \omega$. We observe that `{\em For every finite field $\mathcal{F}=\langle F,...\rangle$, for every nontrivial linearly-ordered vector space $V$ over $\mathcal{F}$, there exists a non-zero linear functional $f: V\rightarrow F$}' $\not\rightarrow AC^{\omega}_{fin}$, $\not\rightarrow LOKW_{4}^{-}$, and $\not\rightarrow AC^{-}_{n}$ in ZFA.  
    
    \item Fix an arbitrary $2\leq n < \omega$. We observe that `{\em For an infinite  graph $G=(V_{G}, E_{G})$ on a linearly-ordered set of vertices $V_{G}$ and a finite graph $H=(V_{H}, E_{H})$, if every finite subgraph of $G$ has a homomorphism into $H$, then so has $G$'} $\not\rightarrow AC^{\omega}_{fin}$, $\not\rightarrow LOKW_{4}^{-}$, and $\not\rightarrow AC^{-}_{n}$ in ZFA.  
    
    \item Fix an arbitrary $2\leq n < \omega$.
    We prove that for every $3\leq k<\omega$, the statement `$\mathcal{P}_{k}$ if the graph $G_{1}$ is on some linearly-orderable set of vertices' $\not\rightarrow AC^{\omega}_{fin}$, $\not\rightarrow LOKW_{4}^{-}$, and $\not\rightarrow AC^{-}_{n}$ in ZFA.  
    
    \item Marshall Hall \cite{Hal1948} proved that if $S$ is a set and $\{S_{i}\}_{i\in I}$ is an indexed family of \textbf{finite} subsets of $S$, then if the following property holds,
    \begin{center}
        (P) for every finite $F\subseteq I$, there is an injective choice function for $\{S_{i}\}_{i\in F}$.
    \end{center}
    then there is an injective choice function for $\{S_{i}\}_{i\in I}$. We abbreviate the above assertion by MHT.
    We recall that BPI implies MHT and MHT implies the Axiom of choice for finite sets ($AC_{fin}$) in ZF (c.f. \cite{HR1998}). Fix an arbitrary $2\leq n < \omega$.  We prove that MHT restricted to a linearly-ordered collection of finite subsets of a set does not imply $AC^{-}_{n}$ in ZFA.
\end{enumerate}

\begin{figure}[!ht]
\centering
\begin{minipage}{\textwidth}
\centering
\begin{tikzpicture}[scale=7]
\draw (-0.6,0.0) node[above] {$CAC^{\aleph_{\alpha}}$};
\draw (-0.6,0.2) node[above] {$DT$};
\draw (-0.6,0.4) node[above] {$CWF$};
\draw (-0.6,0.6) node[above] {$CS$};

\draw (-0.6,1) node[above] {$Statement$ 1};
\draw (-0.6,1.2) node[above] {$Statement$ 2};
\draw (-0.6,1.4) node[above] {$Statement$ 3};
\draw (-0.6,1.6) node[above] {$Statement$ 4};

\draw (0,-0.2) node[above] {$\mathcal{P}_{G,H}$ restricted to finite graph $H$ with 2 vertices, $\mathcal{P}_{3}\not\longrightarrow AC_{3}, AC^{\omega}_{fin}$};

\draw (0.85,0.0) node[above] {There are no amorphous sets};
\draw (0.8,0.4) node[above] {$AC^{-}_{n}$  for every $2\leq n < \omega$};
\draw (0.62,0.8) node[above] {$LOKW^{-}_{4}$};
\draw (0.6,1.2) node[above] {$AC^{\omega}_{fin}$};

\draw[-triangle 60] (-0.5,0.03) -- (0.5,0.03);
\draw (0.33,0.05) -- (0.31,0.0);

\draw[-triangle 60] (-0.5,0.03) -- (0.5,0.43);
\draw (0.33,0.38) -- (0.31,0.33);
\draw (0.33,0.42) -- (0.31,0.37);
\draw (0.33,0.45) -- (0.31,0.4);
\draw (0.33,0.48) -- (0.31,0.43);
\draw (0.33,0.57) -- (0.31,0.52);
\draw (0.33,0.61) -- (0.31,0.56);
\draw (0.33,0.65) -- (0.31,0.60);
\draw (0.33,0.68) -- (0.31,0.63);

\draw[-triangle 60] (-0.5,0.23) -- (0.5,0.43);
\draw (0.33,0.78) -- (0.31,0.73);
\draw (0.33,0.82) -- (0.31,0.77);

\draw[-triangle 60] (-0.5,0.43) -- (0.5,0.43);
\draw (0.33,0.92) -- (0.31,0.87);
\draw (0.33,0.97) -- (0.31,0.92);
\draw (0.33,1.01) -- (0.31,0.96);

\draw (0.33,1.11) -- (0.31,1.06);
\draw (0.33,1.14) -- (0.31,1.09);
\draw (0.33,1.25) -- (0.31,1.2);
\draw (0.33,1.29) -- (0.31,1.24);
\draw (0.33,1.33) -- (0.31,1.28);

\draw[-triangle 60] (-0.5,0.43) -- (0.5,0.83);
\draw[-triangle 60] (-0.5,0.43) -- (0.5,1.23);

\draw[-triangle 60] (-0.5,0.63) -- (0.5,0.43);
\draw[-triangle 60] (-0.5,0.63) -- (0.5,0.83);
\draw[-triangle 60] (-0.5,0.63) -- (0.5,1.23);


\draw[-triangle 60] (-0.45,1.03) -- (0.5,0.43);

\draw[-triangle 60] (-0.45,1.23) -- (0.5,0.43);
\draw[-triangle 60] (-0.45,1.23) -- (0.5,0.83);
\draw[-triangle 60] (-0.45,1.23) -- (0.5,1.23);

\draw[-triangle 60] (-0.45,1.63) -- (0.5,0.43);
\draw[-triangle 60] (-0.45,1.63) -- (0.5,0.83);
\draw[-triangle 60] (-0.45,1.63) -- (0.5,1.23);

\draw[-triangle 60] (-0.45,1.43) -- (0.5,0.43);
\draw[-triangle 60] (-0.45,1.43) -- (0.5,0.83);
\draw[-triangle 60] (-0.45,1.43) -- (0.5,1.23);
\end{tikzpicture}
\end{minipage}
\caption{\em In the above figure, we sketch the results of this note in ZFA. For each regular $\aleph_{\alpha}$, we denote by $CAC^{\aleph_{\alpha}}$ the
statement `if in a partially ordered set, all chains are finite and all antichains have size $\aleph_{\alpha}$, then the
set has size $\aleph_{\alpha}$'.
We use \textbf{Statement 4} to denote 
`$\mathcal{P}_{k}$ for the graph $G_{1}$ on some linearly-orderable set of vertices' for a natural number $k\geq 3$. We use \textbf{Statement 3} to denote `For every finite field $\mathcal{F}=\langle F,...\rangle$, for every nontrivial linearly-ordered vector space $V$ over $\mathcal{F}$, there exists a non-zero linear functional $f: V\rightarrow F$'. 
We use \textbf{Statement 2} to denote `For an infinite  graph $G=(V_{G}, E_{G})$ on a linearly-ordered set of vertices $V_{G}$ and a finite graph $H=(V_{H}, E_{H})$, if every finite subgraph of $G$ has a homomorphism into $H$, then so has $G$'. We use \textbf{Statement 1} to denote Marshall Hall's theorem for linearly-ordered collection of finite subsets of a set.}
\end{figure}
\section{A List of forms and definitions}
\begin{enumerate}
    \item The \textbf{{\em Axiom of Choice}, AC (Form 1 in \cite{HR1998})}: Every family of nonempty sets has a choice function.
    \item  The \textbf{{\em Axiom of Choice for Finite Sets}, $AC_{ﬁn}$ (Form 62 in \cite{HR1998})}: Every family of non-empty ﬁnite sets has a choice function.
    \item \textbf{$AC_{2}$ (Form 88 in \cite{HR1998})}:  Every family of pairs has a choice function. 
    
    \item \textbf{$AC_{n}$ for each $n \in \omega, n \geq 2$ (Form 61 in \cite{HR1998})}: Every family of $n$ element sets has a choice function. We denote by $AC_{n}^{-}$ the statement {\em `Every infinite family of $n$-element sets has a partial choice function' (\textbf{Form 342(n)} in \cite{HR1998}, denoted by $C_{n}^{-}$ in \textbf{Definition 1 (2)} of \cite{HT2019}).} We denote by $LOKW_{n}^{-}$ the statement {\em `Every infinite linearly orderable family $\mathcal{A}$ of $n$-element sets has a partial Kinna--Wegner selection function' (c.f. \textbf{Definition 1 (2)} of \cite{HT2019}).} 
    
    \item \textbf{$AC^{\omega}_{ﬁn}$ (Form 10 in \cite{HR1998})}: Every countably inﬁnite family of non-empty ﬁnite sets has a choice function. We denote by $PAC^{\omega}_{fin}$ the statement {\em `Every countably inﬁnite family of non-empty ﬁnite sets has a partial choice function'}.
    
    \item The \textbf{{\em Principle of Dependent Choice}, DC (Form 43 in \cite{HR1998})}: If $S$ is a relation on a non-empty set $A$ and $(\forall x\in A)(\exists y\in A)(xSy)$ then there is a sequence $a_{0}, a_{1},...$ of elements of $A$ such that $(\forall n\in \omega)(a(n)S a(n+1))$. 
    
    \item \textbf{LW (Form 90 in \cite{HR1998})}: Every linearly-ordered set can be well-ordered.
    \item \textbf{UT(WO, WO, WO) (Form 231 in \cite{HR1998})}: The union of a well-ordered collection of well-orderable sets is well-orderable.
    
    \item \textbf{UT(WO, fin, WO) (Form 10N in \cite{HR1998})}:  The union of a well-orderable family of finite sets is well-orderable. 
    
    \item \textbf{$(\forall\alpha)UT(\aleph_{\alpha},\aleph_{\alpha},\aleph_{\alpha})$ (Form 23 in \cite{HR1998})}: For every ordinal $\alpha$, if $A$ and every member of $A$ has cardinality $\aleph_{\alpha}$, then $\vert \cup A\vert=\aleph_{\alpha}$. 
    
    \item \textbf{UT($\aleph_{0}$, fin, $\aleph_{0}$) (Form 10A in \cite{HR1998})}: The union of a denumerable collection of finite sets is countable.
    
    \item The \textbf{{\em Boolean Prime Ideal Theorem}, BPI (Form 14 in \cite{HR1998})}: Every Boolean algebra has a prime ideal. We recall the following equivalent formulations of BPI.

    \begin{itemize}
        \item \textbf{(Form 14AW in \cite{HR1998})}: The Compactness theorem for propositional logic.
        
        \item The \textbf{{\em Ultrafilter lemma}, UL (Form 14A in \cite{HR1998})}: Every proper filter over a set $S$ in $\mathcal{P}(S)$ can be extended to an ultrafilter.
        
        \item The \textbf{{\em n-coloring theorem for $n\geq 3$}, (Form 14G($n$)($n\in\omega, n\geq 3$) in \cite{HR1998})}: For every graph $G=(V,E)$ if every finite subgraph of $G$ is $n$-colorable then G is $n$-colorable. This is De Bruijn--Erd\H{o}s theorem for $n\geq 3$ colorings.
\end{itemize} 
    \item The \textbf{{\em Principle of consistent choice}, PCC (Form 14AH in \cite{HR1998})}: Let $\mathcal{A}=\{A_{i}\}_{i\in I}$ be a family of finite sets and $\mathcal{R}$ is a symmetric binary relation on $\cup_{i\in I} A_{i}$. Suppose that for every finite $W\subset I$, there is an $\mathcal{R}$-consistent choice function for $\{A_{i}\}_{i\in W}$, then there is an $\mathcal{R}$-consistent choice function for $\{A_{i}\}_{i\in I}$.
    
    We note that Form 14AH in \cite{HR1998} is different than the above formulation. \L{}o\'{s}/Ryll-Nardzewski \cite{LN1951} introduced both the formulations where it was noted that they are equivalent. Let $n\in \omega\backslash \{0,1\}$.  
    We recall the notation $F_{n}$ introduced by Cowen in \cite{Cow1977}, which is PCC restricted to families $\mathcal{A}=\{A_{i}:i\in I\}$, where $\vert A_{i}\vert\leq n$ for all $i\in I$. 
    
    \item \textbf{{\em Marshall Hall's theorem}, MHT (Form 107 in \cite{HR1998})}: If $S$ is a set and $\{S_{i}\}_{i\in I}$ is an indexed family of \textbf{finite} subsets of $S$, then if the following property holds,
    \begin{center}
        (P) for every finite $F\subseteq I$, there is an injective choice function for $\{S_{i}\}_{i\in F}$.
    \end{center}
    then there is an injective choice function for $\{S_{i}\}_{i\in I}$. 
    
    \textbf{Philip Hall's theorem} states that the property (P) is equivalent to the Hall's condition which states that `$\forall F\in [I]^{<\omega}, \vert \cup_{i\in F} S_{i}\vert\geq \vert F \vert$'. We recall that Philip Hall's theorem or finite Hall's theorem can be proved in ZF without using any choice principles.
    
    \item \textbf{A weaker form of \L{}o\'{s}'s lemma, LT (Form 253 in \cite{HR1998})}: If $\mathcal{A}=\langle A, \mathcal{R}^{\mathcal{A}}\rangle$ is a non-trivial relational $\mathcal{L}$-structure over some language $\mathcal{L}$, and $\mathcal{U}$ be an ultrafilter on a non-empty set $I$, then the ultrapower $\mathcal{A}^{I}/\mathcal{U}$ and $\mathcal{A}$ are elementarily equivalent. 
    
    \item \textbf{MCC} (c.f. \textbf{Definition 5} and \textbf{Definition 6} of \cite{Tac2017}): Every topological space with the minimal cover property is compact.
    
    \item \textbf{Bounded and unbounded amorphous sets}: An inﬁnite set $X$ is called {\em amorphous} if $X$ cannot be written as a disjoint union of two inﬁnite subsets. There are two types of amorphous sets, namely bounded amorphous sets and unbounded amorphous sets. Let $\mathcal{U}$ be a finitary partition of an amorphous set $X$. Then all but finitely many elements of $\mathcal{U}$ have the same cardinality, say $n(\mathcal{U})$.   
    Let $\Pi(X)$ be the set of all finitary partitions of $X$ and $n(X)=sup\{n(\mathcal{U}): \mathcal{U} \in \Pi(X)\}$. If $n(X)$ is finite, then $X$ is called {\em bounded amorphous} and if $n(X)$ is infinite, then $X$ is called {\em unbounded amorphous}. We recall \textbf{Theorem 6 of \cite{Tac2017}} which states that MCC $\rightarrow$ ``there are no bounded amorphous sets''.
    
    \item \textbf{(Form 64 in \cite{HR1998})}: There are no amorphous sets.
    
    \item \textbf{Martin's Axiom (c.f. \cite{Tac2016b})}: If $\kappa$ is a well-ordered cardinal, we denote by $MA(\kappa)$ the principle `If $(P,<)$ is a nonempty, c.c.c. quasi order and $\mathcal{D}$ is a family of $\leq\kappa$ dense sets in $P$, then there is a filter $\mathcal{F}$ of $P$ such that $\mathcal{F}\cap D\not=\emptyset$ for all $D\in \mathcal{D}$'. We recall from \textbf{Remark 2.7} of \cite{Tac2016b} that $AC_{fin}^{\omega} + MA(\aleph_{0})\rightarrow$ `for every infinite set X, $2^{X}$ is Baire' and `for every infinite set X, $2^{X}$ is Baire' $\rightarrow$ `there are no amorphous sets'.
    
    \item \textbf{{\em Dilworth's decomposition theorem} for infinite p.o.sets of finite width, DT (c.f. \cite{Tac2019})}:
    If $\mathbb{P}$ is an arbitrary p.o.set, and $k$ is a natural number such that $\mathbb{P}$ has no antichains of size $k + 1$ while at least one $k$-element subset of $\mathbb{P}$ is an antichain, then $\mathbb{P}$ can be partitioned into $k$ chains. We abbreviate the above formulation as DT. We recall \textbf{Theorem 3.1(i)} of \cite{Tac2019}, which states that DT for well-ordered infinite p.o.sets with finite width is provable in ZF.
    
    \item The \textbf{{\em Chain/Antichain Principle}, CAC (Form 217 in \cite{HR1998})}: Every infinite p.o.set has an infinite chain or an infinite antichain. We recall that $CAC$ implies $AC^{\omega}_{fin}$ from \textbf{Lemma 4.4 } of \cite{Tac2019a}.
    
    \item \textbf{CS} (c.f. \cite{THS2016}): Every partially ordered set without a maximal element has two disjoint cofinal subsets.
    \item \textbf{CWF} (c.f. \textbf{Definition 6 (11)} of \cite{Tac2017}): Every partially ordered set has a cofinal well-founded subset.
    
    \item \textbf{Chromatic number of the product of graphs:} We recall a few basic terminologies of graphs. An {\em independent set} is a set of vertices in a graph, no two of which are connected by an edge. A {\em good coloring} of a graph $G=(V_{G},E_{G})$ with a color set $C$ is a mapping $f:V_{G}\rightarrow C$ such that for every  $\{x,y\}\in E_{G}$, $f(x)\not= f(y)$. The {\em chromatic number} $\chi(E_{G})$ of a graph $G=(V_{G},E_{G})$ is the smallest cardinal $\kappa$ such that the graph $G$ can be colored by $\kappa$ colors.
We define the cartesian product of two graphs $G_{1}=(V_{G_{1}},E_{G_{1}})$ and $G_{2}=(V_{G_{2}},E_{G_{2}})$ as the graph 
$G_{1} \times G_{2}=(V_{G_{1} \times G_{2}}, E_{G_{1} \times G_{2}})=(V_{G_{1}}\times V_{G_{2}}, \{\{(x_{0}, x_{1}),(y_{0}, y_{1})\}:\{x_{0}, y_{0}\} \in E_{G_{1}},\{x_{1}, y_{1}\} \in E_{G_{2}}\})$
where $V_{G_{1}}\times V_{G_{2}}$ is the cartesian product of the vertex sets $V_{G_{1}}$ and $V_{G_{2}}$. 
It can be seen that $\chi(E_{G_{1} \times G_{2}}) \leq min (\chi(E_{G_{1}}), \chi(E_{G_{2}}))$. In particular, if $\chi(E_{G_{1}})=k<\omega$ then $\chi(E_{G_{1}\times G_{2}})=k$, since if $f:V_{G_{1}}\rightarrow \{1,...,k\}$ is a good $k$-coloring of $G_{1}$, then $F(\langle x,y\rangle)=f(x)$ is a good $k$-coloring of $G_{1}\times G_{2}$. In \textbf{Theorem 2} of \cite{Haj1985}, Hajnal proved that if $\chi(E_{G_{1}})$ is finite (say $k<\omega$), and $\chi(E_{G_{2}})$ is infinite, then $\chi(E_{G_{1}\times G_{2}})$ is $k$. For a natural number $k<\omega$, we denote by $\mathcal{P}_{k}$ the following statement. 
    \begin{center}
    {\em `$\chi(E_{G_{1}})=k<\omega$ and $\chi(E_{G_{2}})\geq\omega$ implies $\chi(E_{G_{1}\times G_{2}})=k$.'}
    \end{center}
\end{enumerate}

\subsection{Permutation models.} Let $M$ be a model of $ZFA+AC$ where $A$ is a set of atoms or ur-elements. Each permutation $\pi:A\rightarrow A$ extends uniquely to a permutation of $\pi':M\rightarrow M$ by $\epsilon$-induction. Let $\mathcal{G}$ be a group of permutations of $A$ and $\mathcal{F}$ be a normal filter of subgroups of $\mathcal{G}$. For $x\in M$, we denote the symmetric group with respect to $\mathcal {G}$ by $sym_{\mathcal {G}}(x) =\{g\in \mathcal {G} \mid g(x) = x\}$. We say $x$ is {\em $\mathcal{F}$-symmetric} if $sym_{\mathcal{G}}(x)\in\mathcal{F}$ and $x$ is {\em hereditarily $\mathcal{F}$-symmetric} if $x$ is $\mathcal{F}$-symmetric and each element of transitive closure of $x$ is symmetric. We define the permutation model $\mathcal{N}$ with respect to $\mathcal{G}$ and $\mathcal{F}$, to be the class of all hereditarily $\mathcal{F}$-symmetric sets. It is well-known that $\mathcal{N}$ is a model of $ZFA$ (see \textbf{Theorem 4.1} of \cite{Jec1973}). 
If $\mathcal{I}\subseteq\mathcal{P}(A)$ is a normal ideal, then the set $\{$fix$_{\mathcal{G}}E: E\in\mathcal{I}\}$ generates a normal filter over $\mathcal{G}$. Let $\mathcal{I}$ be a normal ideal generating a normal filter $\mathcal{F}_{\mathcal{I}}$ over $\mathcal G$. Let $\mathcal{N}$ be the permutation model determined by $\mathcal{M}, \mathcal{G},$ and $\mathcal{F}_{\mathcal{I}}$. We say $E\in \mathcal{I}$ supports a set $\sigma\in \mathcal{N}$ if fix$_{\mathcal{G}}E\subseteq sym_{\mathcal{G}} (\sigma$). 
\section{Well-ordered structures in ZF}  \subsection{Applications of Loeb's theorem} We recall the following fact from \cite{Ker2000}.

\begin{lem}
{\em (ZF). If $X$ is well-orderable, then $2^{X}$ is compact.}
\end{lem}

\textbf{Remark.} We can also prove \textbf{Lemma 3.1} applying \textbf{Theorem 1} of \cite{Loeb1965}.
\begin{obs}
$UT(WO, fin, WO)$ implies {\em Marshall Hall's theorem for any well-ordered collection of finite subsets of a set}. 
\end{obs}

\begin{proof}
Let $S$ be a set and $\{S_{i}\}_{i\in I}$ be a well-ordered indexed family of \textbf{finite} subsets of $S$ such that the following property holds,
    \begin{center}
        (P) for every finite $F\subseteq I$, there is an injective choice function for $\{S_{i}\}_{i\in F}$.
    \end{center}
We work with the propositional language $\mathcal{L}$ with the following sentence symbols. 
\begin{center}
    $A'_{i,j}$ where $j\in S_{i}$ and $i\in I$.
\end{center}
Let $\mathcal{F}$ be the set of all formulae of $\mathcal{L}$ and $\Sigma\subset \mathcal{F}$ be the collection of the following formulae.

\begin{enumerate}
    \item $\neg (A'_{i,m}\land A'_{j,m})$ for $i\not=j$, $m\in S_{i}\cap S_{j}$.
    \item $\neg (A'_{i,j}\land A'_{i,l})$ for any $l\not=j\in S_{i}$ where $i\in I$.
    \item $A'_{i,y_{1}}\lor A'_{i,y_{2}}...\lor A'_{i,y_{k}}$ for each $i\in I$ where $S_{i}=\{y_{1},... y_{k}\}$. 
\end{enumerate}

We enumerate $Var=\{A'_{i,j}:i\in I, j\in S_{i}\}$ since each $S_{i}$ is finite, $I$ is well-orderable and \textbf{$UT(WO,fin,WO)$ is assumed}.
For every $W \in [I]^{<\omega} \backslash \{\emptyset\}$, we let $\Sigma_{W}$ be the subset of $\mathcal{F}$, which is defined as $\Sigma$ except
that the subscripts in the formulae are from the set $W \cup \bigcup_{i\in W} S_{i}$. 
Endow the discrete 2-element space $\{0,1\}$ with the discrete topology and consider the product space $2^{Var}$ with the product topology.
Let $F_{W} = \{f \in 2^{Var} : \forall\phi\in \Sigma_{W}(f'(\phi)= 1)\}$ where for $f \in 2^{Var}$, the element $f'$ of $2^{\mathcal{F}}$ denotes the valuation mapping determined by $f$.
By \textbf{Philip Hall's theorem} which is provable in ZF without using any choice principles, each $F_{W}$ is non-empty and the family $\mathcal{X}=\{F_{W}: W\in [I]^{<\omega}\backslash\{\emptyset\}\}$ has the finite intersection property. Also for each $W\in [I]^{<\omega}\backslash \{\emptyset\}$, $F_{W}$ is closed in the topological space $2^{Var}$. By \textbf{Lemma 3.1} since $2^{Var}$ is compact in ZF, $\cap \mathcal{X}$ is non-empty.
Pick an $f\in \cap \mathcal{X}$ and let $f'\in 2^{\mathcal{F}}$ be the unique valuation mapping that extends $f$. Clearly, $f'(\phi)=1$ for all $\phi\in \Sigma$. Consequently, we obtain an injective choice function for $\{S_{i}\}_{i\in I}$ by the following claim.

\begin{claim}
{\em If $v$ is a truth assignment which satisfies $\Sigma$, then we can define a system of distinct representatives by

\begin{center}
    $y\in S_{i}$ if and only if $v(A'_{i,y})=T$.
\end{center}
}
\end{claim}

\begin{proof}
By (2) and (3) for each $i\in I$, each collection $S_{i}$ gets assigned a unique representative. By (1), distinct sets $S_{i}$ and $S_{j}$ gets assigned distinct representatives.
\end{proof}
\end{proof}

Banaschewski \cite{Bana1992} proved the uniqueness of the algebraic closure of an arbitrary field applying BPI.\footnote{c.f. the last paragraph of page 384 and page 385 of \cite{Bana1992}.} 

\begin{obs}
$UT(WO, fin, WO)$ implies `{\em If a field $\mathcal{K}$ has an algebraic closure, and the ring of polynomials $\mathcal{K}[x]$ is well-orderable, then the algebraic closure is unique}'.
\end{obs}

\begin{proof}
Let $\mathcal{K}$ be a field, and suppose $\mathcal{E}$ and $\mathcal{F}$ be two algebraic closures of $\mathcal{K}$. We prove that there is an isomorphisn from $\mathcal{E}$ onto $\mathcal{F}$ which fix $\mathcal{K}$ pointwise. Let $\mathcal{E}_{u}$ and $\mathcal{F}_{u}$ be the splitting fields of $u\in \mathcal{K}[x]$ inside $\mathcal{E}$ and $\mathcal{F}$ respectively. Let $\mathcal{H}_{u}$ be the set of all isomorphisms from $\mathcal{E}_{u}$ onto $\mathcal{F}_{u}$ which fix $\mathcal{K}$. Clearly, $\mathcal {H}_{u}$ is a non-empty, finite set. Also, we can see that $\cup_{u} \mathcal{E}_{u}=\mathcal{E}$ and $\cup_{u} \mathcal{F}_{u}=\mathcal{F}$.
Let $\mathcal{H}=\Pi_{u\in \mathcal{K}[x]} \mathcal{H}_{u}$, and if $v\vert w$ define $H_{v,w}=\{(h_{u})\in \mathcal{H}: h_{v}=h_{w}\restriction E_{v}\}$. Clearly, $H_{v,w}$ has finite intersection property and they are closed in the product topology of $\mathcal{H}$, where each $\mathcal{H}_{u}$ is discrete. Since \textbf{$\mathcal{K}[x]$ is well-orderable} as assumed and  
for each $u\in\mathcal{K}[x]$, $\mathcal{H}_{u}$ is finite, we have that $\cup_{u\in\mathcal{K}[x]} \mathcal{H}_{u}$ is well-orderable by $UT(WO, fin, WO)$. By \textbf{Theorem 1} of \cite{Loeb1965}, $\mathcal{H}$ is compact. Consequently,  $\bigcap_{v\vert w}\mathcal{H}_{v,w}\not=\emptyset$ and each $(h_{u})$ in this intersection determines a unique embedding $h:\cup_{u} \mathcal{E}_{u}\rightarrow \cup_{u}\mathcal{F}_{u}$ which is onto and fixes $\mathcal{K}$.
\end{proof}


\begin{obs}
The statement `{\em For an infinite graph $G=(V_{G}, E_{G})$ on a well-ordered set of vertices $V_{G}$ and a finite graph $H=(V_{H}, E_{H})$, if every finite subgraph of $G$ has a homomorphism into $H$, then so has $G$}' is provable in ZF.
\end{obs}
\begin{proof}
Fix a finite graph $H=(V_{H}, E_{H})$ and a graph $G=(V_{G}, E_{G})$ on a well-ordered set of vertices $V_{G}$. We consider $V_{H}=\{v_{1},...v_{k}\}$ for some $k<\omega$. We work with the propositional language $\mathcal{L}$ with the following sentence symbols. 
\begin{center}
    $A'_{x_{i},v_{j}}$ where $v_{j}\in V_{H}$ and $x_{i}\in V_{G}$.
\end{center}
Let $\mathcal{F}$ be the set of all formulae of $\mathcal{L}$ and $\Sigma\subset \mathcal{F}$ be the collection of the following formulae.

\begin{enumerate}
    \item $A_{x_{i},v_{m}}'\land A_{x_{j},v_{l}}'$ if and only if $\{x_{i},x_{j}\}\in E_{G}$ implies $\{v_{m},v_{l}\}\in E_{H}$.
    \item $\neg (A_{x_{i},v_{j}}'\land A_{x_{i},v_{l}}')$ for any $v_{l},v_{j}\in V_{H}$ such that $v_{l}\not=v_{j}$ and each $x_{i}\in V_{G}$.
    \item $A_{x_{i},v_{1}}'\lor A_{x_{i},v_{2}}'...\lor A_{x_{i},v_{k}}'$ for each $x_{i}\in V_{G}$. 
\end{enumerate}

By our assumption $V_{G}$ is well-orderable and $V_{H}$ is finite. So $V_{H}$ is well-orderable. Consequently, $V_{G}\times V_{H}$ is well-orderable in ZF.
We enumerate $Var=\{A'_{x_{i},v_{j}}:x_{i}\in V_{G}, v_{j}\in V_{H}\}$.
By assumption, for every $s\in[V_{G}]^{<\omega}$ there is a homomorphism $f_{s}:G\restriction s\rightarrow H$ of $G\restriction s$ into $H$. 
Following the methods used in the proof of \textbf{Observation 3.2}, we may obtain a $f'\in 2^{\mathcal{F}}$ such that $f'(\phi)=1$ for all $\phi\in \Sigma$. Consequently, we can obtain a homomorphism $h$ from $G$ to $H$.
\end{proof}

\begin{obs}
The statement `{\em For every finite field $\mathcal{F}=\langle F,...\rangle$, for every nontrivial well-ordered vector space $V$ over $\mathcal{F}$, there exists a non-zero linear functional $f: V\rightarrow F$}' is provable in ZF.
\end{obs}

\begin{proof}
We follow the proof of \textbf{Theorem 18 of \cite{HT2013}} and modify it in the context of well-orderable vector space. Fix a finite field $\mathcal{F}=\langle F,...\rangle$ where $F=\{v_{1},...v_{k}\}$ and a nontrivial well-ordered vector space $V$ over $\mathcal{F}$. We work with the propositional language $\mathcal{L}$ with the following sentence symbols. 
\begin{center}
    $A'_{x_{i},v_{j}}$ where $v_{j}\in F$ and $x_{i}\in V$.
\end{center}
Let $\mathcal{F}'$ be the set of all formulae of $\mathcal{L}$ and $\Sigma\subset \mathcal{F}'$ be the collection of the following formulae.

\begin{enumerate}
    \item $A'_{a,1}$.
    \item $A'_{x_{i},v_{j}}\rightarrow A'_{v_{k} x_{i},v_{k}v_{j}}$ for $v_{k}, v_{j}\in F$ and $x_{i}\in V$.
    \item $A'_{x_{i},v_{j}}\land A'_{x_{i'},v_{j'}}\rightarrow A'_{x_{i}+x_{i'}, v_{j}+v_{j'}}$ for $x_{i},x_{i'}\in V$ and $v_{j}, v_{j'}\in F$.
    \item $\neg (A'_{x_{i},v_{j}}\land A'_{x_{i},v_{l}})$ for any $v_{l},v_{j}\in F$ such that $v_{l}\not=v_{j}$ and each $x_{i}\in V$.
    \item $A_{x_{i},v_{1}}'\lor A_{x_{i},v_{2}}'...\lor A_{x_{i},v_{k}}'$ for each $x_{i}\in V$. 
\end{enumerate}

By our assumption $V$ is well-orderable and $F$ is finite. So $F$ is well-orderable. Consequently, $V\times F$ is well-orderable.
We enumerate $Var=\{A'_{x_{i},v_{j}}:x_{i}\in V, v_{j}\in F\}$.
Fix $V'\in [V]^{<\omega}$. Let $W$ be the subspace of $V$ generated by the finite set $V'\cup \{a\}$. We can see that $W$ is finite since $F$ is finite. Consequently, a linear functional $f:W\rightarrow F$ with $f(a)=1$ can be constructed in ZF. 
Following the methods used in the proof of \textbf{Observation 3.2}, we can obtain a non-zero linear functional $f:V\rightarrow F$.
\end{proof}

\begin{obs}
For every $3\leq k<\omega$, the statement $\mathcal{P}_{k}$ for the graph $G_{1}$ on some well-orderable set of vertices is provable under ZF.
\end{obs}
\begin{proof}
Fix $3\leq k<\omega$. Suppose $ \chi(E_{G_{1}})=k$, $\chi(E_{G_{2}})\geq\omega$ and $G_{1}$ is a graph on some well-orderable set of vertices. First we observe that if $g:V_{G_{1}}\rightarrow \{1,...,k\}$ is a good $k$-coloring of $G_{1}$, then $G(\langle x,y\rangle)=g(x)$ is a good $k$-coloring of $G_{1}\times G_{2}$. So, $\chi(E_{G_{1}\times G_{2}})\leq k$. For the sake of contradiction assume that $F:V_{G_{1}}\times V_{G_{2}}\rightarrow \{1,...,k-1\}$ is a good coloring of $G_{1}\times G_{2}$. 
For each color $c\in \{1,...,k-1\}$ and each vertex $x\in V_{G_{1}}$ we let $A_{x,c}=\{y\in V_{G_{2}}:F(x,y)=c\}$. 

\begin{claim}(ZF).
{\em For all finite $F\subset V_{G_{1}}$, there exists a mapping $i_{F}:F\rightarrow\{1,...,k-1\}$ such that for any $x,x'\in F$, $A_{x,i_{F}(x)}\cap A_{x',i_{F}(x')}$ is not independent.}
\end{claim}

\begin{proof}
Since any superset of non-independent set is non-independent, it is enough to show that for all finite $F\subset V_{G_{1}}$, there exists an $i_{F}:F\rightarrow\{1,...,k-1\}$ such that $\cap_{x\in F} A_{x,i_{F}(x)}$ is not independent. For the sake of contradiction assume that there exist a finite $F\subset V_{G_{1}}$ such that for all $i_{F}:F\rightarrow\{1,...,k-1\}$, $\cap_{x\in F} A_{x,i_{F}(x)}$ is independent. Now, $V_{G_{2}}=\cup_{i_{F}:F\rightarrow\{1,...,k-1\}} \cap_{x\in F} A_{x,i_{F}(x)}$. Thus $V_{G_{2}}$ can be written as a finite union of independent sets which contradicts the fact that $\chi(E_{G_{2}})$ is infinite. Thus for all finite $F\subset V_{G_{1}}$, we can obtain a mapping $i_{F}:F\rightarrow \{1,...,k-1\}$ such that $\cap_{x\in F} A_{x,i_{F}(x)}$ is not independent.
\end{proof}

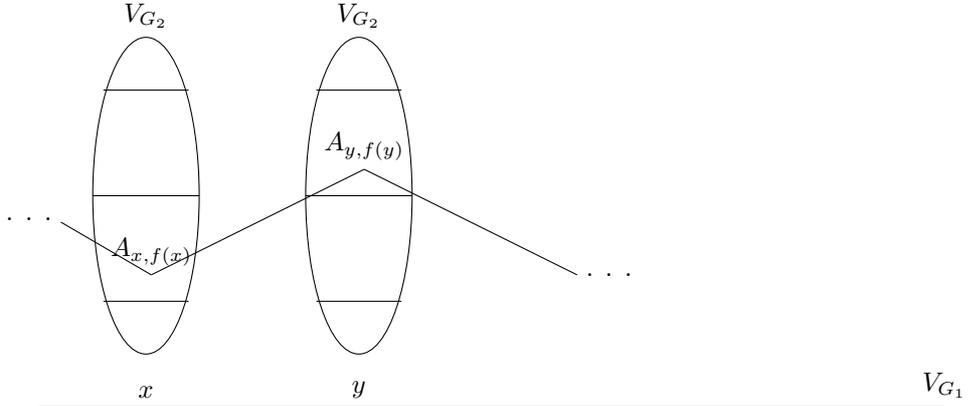
\begin{figure}[!ht]
\centering
\begin{minipage}{\textwidth}
\centering
\begin{tikzpicture}[scale=7]
\draw (-0.2,0.0) -- (1.5,0);
\draw (1.5,0) node[above] {$V_{G_{1}}$};
\draw (0,0.4) ellipse (0.1cm and 0.3cm);
\draw (0,0) node[above] {$x$};
\draw (0.4,0) node[above] {$y$};
\draw (-0.08,0.20) -- (0.08,0.20);
\draw (-0.1,0.40) -- (0.1,0.40);
\draw (-0.08,0.60) -- (0.08,0.60);
\draw (0,0.70) node[above] {$V_{G_{2}}$};
\draw (0.01,0.25) node[above] {$A_{x,f(x)}$};
\draw (0.01,0.25) -- (-0.16,0.35);
\draw (-0.22,0.335) node[above] { . . .};
\draw (0.4,0.4) ellipse (0.1cm and 0.3cm);
\draw (0.32,0.20) -- (0.48,0.20);
\draw (0.3,0.40) -- (0.5,0.40);
\draw (0.32,0.60) -- (0.48,0.60);
\draw (0.41,0.45) node[above] {$A_{y,f(y)}$};
\draw (0.01,0.25) -- (0.41,0.45);
\draw (0.4,0.70) node[above] {$V_{G_{2}}$};
\draw (0.81,0.25) -- (0.41,0.45);
\draw (0.87,0.23) node[above] { . . .};
\end{tikzpicture}
\end{minipage}

\caption{\em A map $f:V_{G_{1}}\rightarrow \{1,...,k-1\}$ such that intersection of any two elements in $\{A_{x,f(x)}:x\in V_{G_{1}}\}$ is not independent.}
\end{figure}

Endow $\{1,2,...,k-1\}$ with the discrete topology. Since $V_{G_{1}}$ is well-orderable,  $\{1,2,...,k-1\}\times V_{G_{1}}$ is well-orderable under ZF. Applying \textbf{Theorem 1} of \cite{Loeb1965}, $\{1,2,...,k-1\}^{V_{G_{1}}}$ is compact.
For $s\in [V_{G_{1}}]^{<\omega}$, define $F_{s}=\{f\in \{1,2,...,k-1\}^{V_{G_{1}}}: x,y\in s, x\not=y\rightarrow A_{x,f(x)}\cap A_{y,f(y)}$ is not independent$\}$. By \textbf{claim 3.8}, for each $s\in [V_{G_{1}}]^{<\omega}$ we have that $F_{s}$ is non-empty. We can see that $\{F_{s}:s\in [V_{G_{1}}]^{<\omega}\}$ has finite intersection property as $F_{s_{0}\cup...\cup s_{k}}\subseteq F_{s_{0}}\cap ... \cap F_{s_{k}}$. Thus by compactness of $\{1,2,...,k-1\}^{V_{G_{1}}}$, there is a $f\in \bigcap\{F_{s}:s\in [V_{G_{1}}]^{<\omega}\}$.
Clearly, for any $x,x'\in V_{G_{1}}$, $A_{x,f(x)}\cap A_{x',f(x')}$ is not independent (see \textbf{figure 2}). 
Since $x\rightarrow f(x)$ is not a good coloring in $G_{1}$ as $\chi(E_{G_{1}})=k$, there are $x,x'\in V_{G_{1}}$ with $f(x)=f(x')=j$ and $\{x,x'\}\in E_{G_{1}}$. Consequently, $A'=A_{x,f(x)}\cap A_{x',f(x')}$ is not independent. Pick $y,y'\in A'$ joined by an edge in $E_{G_{2}}$. Then $(x,y)$ and $(x',y')$ are joined in $E_{G_{1}}\times E_{G_{2}}$ and get the same color $j$ which is a contradiction to the fact that $F$ is a good coloring of $G_{1}\times G_{2}$. 
\end{proof}

\subsection{On partially ordered sets based on a well-ordered set of elements} The first author modifies the arguments from \textbf{Claim 5} of \cite{Tac2016} and observes the following.  

\begin{obs} The following holds.
\begin{enumerate}
    \item 
$UT(\aleph_{0}, \aleph_{0},
    \aleph_{0})$ implies {\em `If in a partially ordered set based on a well-ordered set of elements, all chains are finite and all antichains are countable, then the set is countable'}.
    
    \item $UT(\aleph_{\alpha}, \aleph_{\alpha},\aleph_{\alpha})$ implies {\em `If in a partially ordered set based on a well-ordered set of elements, all chains are finite and all antichains have size $\aleph_{\alpha}$, then the set has size $\aleph_{\alpha}$'} for any regular $\aleph_{\alpha}$.
\end{enumerate}
\end{obs}

\begin{proof}
We prove \textbf{Observation 3.9 (1)}. For the sake of contradiction, assume that all antichains of $(P,\leq)$ are countable, all chains of $(P,\leq)$ are finite, but the set $P$ is uncountable and well-ordered. We construct an infinite chain in $(P,\leq)$ using $UT(\aleph_{0}, \aleph_{0}, \aleph_{0})$ and obtain the desired contradiction.

\begin{claim}
{\em $\leq$ is a well-founded relation on $P$ i.e., every non-empty subset of $P$ has a $\leq$-minimal element.}
\end{claim}

\begin{proof}
Let $P$ is well-orderable, say by $\preceq$. We claim that $\leq$ is a well-founded relation on $P$. Otherwise, there is a nonempty subset $P_{1}\subseteq P$ with no minimal elements. Consequently, using the fact that $\preceq$ is a well-ordering in $P$, we can obtain a strictly $\leq$-decreasing sequence of elements of $P_{1}$. This contradicts the assumption that $P$ has no infinite chains.
\end{proof}

Without loss of generality we may assume $P=\cup\{P_{\alpha}:\alpha<\kappa\}$ where $\kappa$ is a well-ordered cardinal, $P_{0}$ is the set of minimal elements of $P$ and for each $\alpha<\kappa$, $P_{\alpha}$ is the set of minimal elements of $P\backslash \cup\{P_{\beta}:\beta<\alpha\}$. For each $\alpha<\kappa$, $P_{\alpha}$ is countable since $P_{\alpha}$ is an antichain.
\begin{itemize}
    \item We note that $P=\cup\{P_{p}:p\in P_{0}\}$ where $P_{p}=\{q\in P:p\leq q\}$. Since $P$ is uncountable and $P_{0}$ is countable, $P_{p}$ is uncountable for some $p\in P_{0}$. Otherwise for all $p\in P_{0}$, $P_{p}$ is either countable or finite and $UT(\aleph_{0}, \aleph_{0}, \aleph_{0})$ + $UT(\aleph_{0}, fin, \aleph_{0})$ implies $P$ is countable which is a contradiction.
    Now $UT(\aleph_{0}, \aleph_{0}, \aleph_{0})$ implies $UT(\aleph_{0}, fin, \aleph_{0})$ in ZF, thus $UT(\aleph_{0}, \aleph_{0}, \aleph_{0})$ suffices.
    Since $\{q\in P_{0}:P_{q}$ is uncountable$\}$ is a non-empty subset of $P$, we can find a least $p_{0}\in P_{0}$ with respect to $\preceq$ such that $P_{p_{0}}$ is uncountable.
    
    \item Let us consider $P'=P_{p_{0}}\backslash \{p_{0}\}$. So, $P'$ is uncountable. Again if $P'_{1}$ is the set of minimal elements of $P'$, we can write $P'=\cup\{P_{p}:p\in P'_{1}\}$ where $P_{p}=\{q\in P:p\leq q\}$.
    Since $P'$ is uncountable and $P'_{1}$ is countable (since all antichains of $(P,\leq)$ are countable by assumption), once again applying $UT(\aleph_{0}, \aleph_{0}, \aleph_{0})$ as in the previous paragraph, $P_{p}$ is uncountable for some $p\in P'_{1}$. 
    Since $\{q\in P'_{1}:P_{q}$ is uncountable$\}$ is a non-empty subset of $P$, we can find a least $p_{1}\in P'_{1}$ with respect to $\preceq$ such that $P_{p_{1}}$ is uncountable.
    We can see that $p_{0}<p_{1}$. 
\end{itemize}

Continuing this process step by step we obtain a sequence $\langle p_{n}:n\in\omega\rangle$ of elements of $P$ such that $p_{n}<p_{n+1}$ for each $n\in \omega$. Consequently, we obtain an infinite chain.

\textbf{Remark.} Similarly for any regular $\aleph_{\alpha}$, assuming $UT(\aleph_{\alpha},\aleph_{\alpha},\aleph_{\alpha})$ we can prove the following since alephs are well-ordered. {\em `If in a partially ordered set based on well-ordered set of elements, all chains are finite and all antichains have size $\aleph_{\alpha}$, then the set has size $\aleph_{\alpha}$.'} Consequently, we can prove \textbf{Observation 3.9(2)}.
\end{proof}

\section{Consistency results}
\begin{thm}
{\em For every natural number $n\geq 2$, there is a permutation model $\mathcal{N}$ of $ZFA$ where
 CAC holds and $AC_{n}^{-}$ fails. Moreover, we can observe the following in the model.
\begin{enumerate}
    \item CS, as well as CWF, holds.
    \item DT holds, $MA(\aleph_{0})$ fails, MCC fails, LT fails. 
    \item If in a partially ordered set, all chains are finite and all antichains are countable, then the set is countable. Moreover, if in a partially ordered set, all chains are finite and all antichains have size $\aleph_{\alpha}$, then the set has size $\aleph_{\alpha}$ for any regular $\aleph_{\alpha}$.
    \item The following statements hold.
    \begin{enumerate}
        \item Marshall Hall's theorem for linearly-ordered collection of finite subsets of a set.
        \item For every $3\leq k<\omega$,  $\mathcal{P}_{k}$ holds for any graph $G_{1}$ on some linearly-orderable set of vertices.
        \item For an infinite  graph $G=(V_{G}, E_{G})$ on a linearly-ordered set of vertices $V_{G}$ and a finite graph $H=(V_{H}, E_{H})$, if every finite subgraph of $G$ has a homomorphism into $H$, then so has $G$.
        \item For every finite field $\mathcal{F}=\langle F,...\rangle$, for every nontrivial linearly-orderable vector space $V$ over $\mathcal{F}$, there exists a non-zero linear functional $f:V\rightarrow F$.
    \end{enumerate}
\end{enumerate}
}
\end{thm}

\begin{proof}
In \textbf{Theorem 8} of \cite{HT2019}, Halbeisen and Tachtsis constructed a permutation model $\mathcal{N}$ where for arbitrary $n\geq 2$, $AC_{n}^{-}$ fails but CAC holds. We fix an arbitrary integer $n\geq 2$ and recall the model constructed in the proof of \textbf{Theorem 8} of \cite{HT2019} as follows.

\begin{itemize}
    \item \textbf{Defining the ground model $M$.} 
We start with a ground model $M$ of $ZFA+AC$ where $A$ is a countably infinite set of atoms written as a disjoint union $\cup\{A_{i}:i\in \omega\}$ where for each $i\in \omega$, $A_{i}=\{a_{i_{1}},a_{i_{2}},... a_{i_{n}}\}$. 
    \item \textbf{Defining the group $\mathcal{G}$  of permutations and the filter $\mathcal{F}$ of subgroups of $\mathcal{G}$.}
    \begin{itemize}
    \item \textbf{Defining $\mathcal{G}$.} $\mathcal{G}$ is defined in \cite{HT2019} in a way so that \textbf{if $\eta\in \mathcal{G}$, then $\eta$ only moves finitely many
    atoms} and for all $i\in \omega$, $\eta(A_{i})=A_{k}$ for some $k\in \omega$. We recall the details from \cite{HT2019} as follows.
    For all $i\in \omega$, let $\tau_{i}$ be the $n$-cycle $a_{i_{1}}\mapsto a_{i_{2}}\mapsto ... a_{i_{n}}\mapsto a_{i_{1}}$. For every permutation $\psi$ of $\omega$, which moves only finitely many natural numbers, let $\phi_{\psi}$ be the permutation of $A$ defined by $\phi_{\psi}(a_{i_{j}})=a_{\psi(i)_{j}}$ for all $i\in \omega$ and $j=1,2,...,n$. Let $\eta\in \mathcal{G}$ if and only if $\eta=\rho \phi_{\psi}$ where $\psi$ is a permutation of $\omega$ which moves only finitely many natural numbers and $\rho$ is a permutation of $A$ for which there is a finite $F\subseteq\omega$ such that for every $k\in F$, $\rho\restriction A_{k}=\tau^{j}_{k}$ for some $j<n$, and $\rho$ fixes $A_{m}$ pointwise for every $m\in \omega\backslash F$. 
    \item \textbf{Defining $\mathcal{F}$.} Let $\mathcal{F}$ be the filter of subgroups of $\mathcal{G}$ generated by $\{$fix$_{\mathcal{G}}(E): E\in [A]^{<\omega}\}$. 
    \end{itemize}
    \item \textbf{Defining the permutation model.} Consider the permutation model $\mathcal{N}$ determined by $M$, $\mathcal{G}$ and $\mathcal{F}$.
\end{itemize}

Following \textbf{point 1} in the proof of \textbf{Theorem 8} of \cite{HT2019}, both $A$ and $\mathcal{A}=\{A_{i}\}_{i\in \omega}$ are amorphous in $\mathcal{N}$ and no infinite subfamily $\mathcal{B}$ of $\mathcal{A}$ has a Kinna--Wegner selection function. Consequently, $AC^{-}_{n}$ fails. The first author observes the following.

\begin{lem}
{\em In $\mathcal{N}$, DT, CS as well as CWF holds. Moreover the following holds in $\mathcal{N}$.
\begin{itemize}
    \item If in a partially ordered set, all chains are finite and all antichains are countable, then the set is countable.
    \item `If in a partially ordered set, all chains are finite and all antichains have size $\aleph_{\alpha}$, then the set has size $\aleph_{\alpha}$' for any regular $\aleph_{\alpha}$.
\end{itemize}
}
\end{lem}
\begin{proof}
We follow the steps below.

\begin{enumerate}
\item
Let $(P,\leq)$ be a p.o.set in $\mathcal{N}$ and $E\in [A]^{<\omega}$ be a support of $(P,\leq)$. We can write $P$ as a disjoint union of ﬁx$_{\mathcal{G}}(E)$-orbits, i.e., $P=\bigcup \{Orb_{E}(p):p 
\in P\}$, where $Orb_{E}(p)=\{\phi(p): \phi\in$ ﬁx$_{\mathcal{G}}(E)\}$ for all $p \in P$. The family $\{Orb_{E}(p): p \in P\}$ is well-orderable in $\mathcal{N}$ since ﬁx$_{\mathcal{G}}(E) \subseteq Sym_{\mathcal{G}}(Orb_{E}(p))$ for all $p \in P$. 
 
\item Since \textbf{if $\eta\in \mathcal{G}$, then $\eta$ only moves finitely many atoms}, $Orb_{E}(p)$ is an antichain in $P$ for each $p\in P$. Otherwise there is a $p\in P$, such that $Orb_{E}(p)$ is not an antichain in $(P,\leq)$. Thus, for some $\phi,\psi\in$ fix$_{\mathcal{G}}(E)$, $\phi(p)$ and $\psi(p)$ are comparable. Without loss of generality we may assume $\phi(p)<\psi(p)$. Since \textbf{if $\eta\in \mathcal{G}$, then $\eta$ only moves finitely many atoms}, there exists some $k<\omega$ such that $\phi^{k}=1_{A}$. Let $\pi=\psi^{-1}\phi$. Consequently, $\pi(p)<p$ and $\pi^{k}=1_{A}$ for some $k\in\omega$. Thus, $p=\pi^{k}(p)<\pi^{k-1}(p)<...<\pi(p)<p$. By transitivity of $<$, $p<p$, which is a contradiction.

\item We prove that in $\mathcal{N}$, $DT$ holds. Let $E\subset A$ be a finite support of an infinite p.o.set $\mathbb{P}=(P,<)$ with finite width. Then $P=\bigcup\{Orb_{E}(p): p\in P\}$. Following (2), $Orb_{E}(p)$ is an antichain in $\mathbb{P}$.
Consequently, $ Orb_{E}(p)$ is finite for each $p\in P$ since the width of $\mathbb{P}$ is finite. Following (1), $\{Orb_{E}(p): p\in P\}$ is well-orderable in $\mathcal{N}$. Following \textbf{point 4} in the proof of \textbf{Theorem 8} of \cite{HT2019} and \textbf{Lemma 3} of \cite{Tac2016}, UT(WO,WO,WO) holds in $\mathcal{N}$, and so $P$ is well-orderable in $\mathcal{N}$. Applying \textbf{Theorem 3.1(i) of \cite{Tac2019}}, DT holds in $\mathcal{N}$.

 \item To see that CS as well as CWF holds in $\mathcal{N}$ we follow \textbf{Theorem 3.26} of \cite{THS2016} and \textbf{Theorem 10(ii) of \cite{Tac2017}} respectively. We sketch the important steps below. 
 \begin{enumerate}
     \item We follow \textbf{Theorem 3.26} of \cite{THS2016} to see that CS holds in $\mathcal{N}$ as follows. Let $(P,\leq)$ be a poset without maximal elements supported by E. Following (1), $\mathcal{O}=\{Orb_{E}(p): p\in P\}$ is a well-ordered partition of $P$. Define $\preceq$ on $\mathcal{O}$, as $X\preceq Y\leftrightarrow \exists x\in X, \exists y\in Y$ such that $x\leq y$. Since $(P,\leq)$ has no maximal element, $(\mathcal{O},\preceq)$ has no maximal element following (2). 
     Since $\mathcal{O}$ is well-ordered there exists a partition $\mathcal{U}_{\mathcal{O}}=\{\mathcal{Q},\mathcal{R}\}$ of $\mathcal{O}$ in 2 cofinal subsets. Consequently, $\mathcal{U}_{P}=\{\cup\mathcal{Q},\cup\mathcal{R}\}$ is a partition of $P$ in 2 cofinal subsets.
     
     \item We follow \textbf{Theorem 10 (ii)} of \cite{Tac2017} to see that CWF holds in $\mathcal{N}$ as follows. Let $(P,\leq)$ be a poset supported by $\mathcal{N}$. Since $\mathcal{O}=\{Orb_{E}(p): p\in P\}$ is well-orderable, it has a cofinal well-founded subset $\mathcal{W}=\{W_{\alpha}:\alpha<\gamma\}$ such that for $\beta<\alpha$, $W_{\alpha}\not\preceq W_{\beta}$ for all $\beta,\alpha<\gamma$. Consequently, $C=\cup \mathcal{W}$ is a cofinal well-founded subset of $P$. 
 \end{enumerate}
 
 \item We show the following in $\mathcal{N}$.
\begin{center}
    {\em `If in a partially ordered set, all chains are finite and all antichains are countable, then the set is countable.'}
\end{center}
It is known that in every $FM$-model $UT(WO,WO,WO)$ implies $(\forall \alpha)UT(\aleph_{\alpha}, \aleph_{\alpha}, \aleph_{\alpha})$ (c.f. page 176 of \cite{HR1998}). Consequently, $UT(\aleph_{0}, \aleph_{0}, \aleph_{0})$ holds in $\mathcal{N}$. Let $(P,<)$ be an uncountable p.o.set in $\mathcal{N}$ where all antichains are countable and $E\in [A]^{<\omega}$ be a support of $(P,<)$. Following (1), $\mathcal{O}=\{Orb_{E}(p): p\in P\}$ is a well-ordered partition of $P$ since for all $p\in P$, $E$ is a support of $Orb_{E}(p)$.
Following (2), $Orb_{E}(p)$ is an antichain and hence countable. Consequently, $Orb_{E}(p)$ is well-orderable. Since $UT(WO,WO,WO)$ holds in $\mathcal{N}$, $P$ is well-orderable. By \textbf{Observation 3.9(1)}, since $UT(\aleph_{0}, \aleph_{0}, \aleph_{0})$ holds in $\mathcal{N}$, there is an infinite chain in $\mathcal{N}$. 
\item Following (5) and \textbf{Observation 3.9(2)}, we can prove the following in $\mathcal{N}$. 

\begin{center}
{\em `If in a partially ordered set $(P,<)$, all chains are finite and all antichains have size $\aleph_{\alpha}$, then the set has size $\aleph_{\alpha}$'.} 
\end{center}
\end{enumerate}
\end{proof}

\textbf{Remark.} The referee pointed out that the statements ‘If in a partially ordered set based on
a well-ordered set of elements all chains are finite and all antichains are countable then the set is
countable’ and ‘If in a partially ordered set based on a well-ordered set of elements all chains are
finite and all antichains have size $\aleph_{\alpha}$ then the set has size $\aleph_{\alpha}$’ are true in all Fraenkel-Mostowski
permutation models. So \textbf{Observation 3.9(1)} and \textbf{Observation 3.9(2)} are not needed in the proofs of parts (5) and (6) of the proof of \textbf{Lemma 4.2}.

\begin{lem}
{\em In $\mathcal{N}$, $MA(\aleph_{0})$ fails.}
\end{lem}

\begin{proof}
Since $A$ is amorphous, the statement `for all infinite $X$, $2^{X}$ is Baire' is false following \textbf{Remark 2.7} of \cite{Tac2016b}. Since CAC holds in $\mathcal{N}$, $AC_{fin}^{\omega}$ holds as well (c.f. \textbf{Lemma 4.4} of \cite{Tac2019a}). Consequently, $MA(\aleph_{0})$ fails following \textbf{Remark 2.7} of \cite{Tac2016b}.
\end{proof}

\begin{lem}
{\em In $\mathcal{N}$, $MCC$ fails.}
\end{lem}

\begin{proof}
Modifying the proof of \textbf{Theorem 8 (ii) of \cite{Tac2017}}, we can see that $n(A)= n$. Thus there is a bounded amorphous set A. Consequently, MCC fails by \textbf{Theorem 6} of \cite{Tac2017}.
\end{proof}

\begin{lem}
{\em In $\mathcal{N}$, LT fails.}
\end{lem}

\begin{proof}
Since $\mathcal{A}$ is an amorphous set of non-empty sets which has no choice function in $\mathcal{N}$, following \textbf{Lemma 4.1}(i)\cite{Tac2019a}, LT fails in $\mathcal{N}$.
\end{proof}

\begin{lem}
{\em In $\mathcal{N}$, the following statements hold for linearly-ordered structures.
\begin{enumerate}
        \item Marshall Hall's theorem for linearly-ordered collection of finite subsets of a set.
        \item For every $3\leq k<\omega$,  $\mathcal{P}_{k}$ holds for any graph $G_{1}$ on some linearly-orderable set of vertices.
        \item For an infinite  graph $G=(V_{G}, E_{G})$ on a linearly-ordered set of vertices $V_{G}$ and a finite graph $H=(V_{H}, E_{H})$, if every finite subgraph of $G$ has a homomorphism into $H$, then so has $G$.
        
        \item For every finite field $\mathcal{F}=\langle F,...\rangle$, for every nontrivial linearly-orderable vector space $V$ over $\mathcal{F}$, there exists a non-zero linear functional $f:V\rightarrow F$.
\end{enumerate}
}    
\end{lem}

\begin{proof}
Since $UT(WO,WO,WO)$ and LW holds in $\mathcal{N}$ (c.f. pt 4 and pt 3 in the proof of \textbf{Theorem 8} in \cite{HT2019}), (1), (2), (3) and (4) hold in $\mathcal{N}$ following the observations in \textbf{section 3}. 
\end{proof}
\end{proof}
\textbf{Remark 1.} In \textbf{Theorem 7} of \cite{Tac2019a}, Tachtsis generalized the above construction and proved that $AC^{LO}+LW\not\rightarrow LT$ by constructing a permutation model $\mathcal{N}$. Since $AC^{WO}$ holds in $\mathcal{N}$, $DC$ holds in $\mathcal{N}$ as well (c.f. \textbf{Theorem 8.2} of \cite{Jec1973}). We observe another standard argument to see that $DC$ holds in $\mathcal{N}$.
Since $\mathcal{I}$ is closed under countable unions in the model, we can see that DC holds in $\mathcal{N}$. Let $\mathcal{R}$ is a relation in $\mathcal{N}$ such that if $x\in dom(\mathcal{R})$, there exists a $y$ such that $xRy$. Consequently, there is a sequence $\langle x_{n}:n\in\omega\rangle$ in the ground model $M$ such that for each $n\in\omega$, $x_{n}Rx_{n+1}$. If $x_{n}$ is supported by $E_{n}$ for every $n\in\omega$, then $\langle x_{n}:n\in \omega\rangle$ is supported by $\cup_{n\in\omega} E_{n}$. Since $\mathcal{I}$ is closed under countable unions, the sequence $\langle x_{n}:n\in\omega\rangle$ is in $\mathcal{N}$. 

A class of models $M_{\aleph_{\alpha}}$ for any regular cardinal $\aleph_{\alpha}$ (similar to the model $M_{\aleph_{1}}$ constructed in \textbf{Theorem 7} of \cite{Tac2019a}) can be defined where $AC^{LO}$ and $LW$ holds but $LT$ fails, by replacing $\aleph_{1}$ by $\aleph_{\alpha}$. Moreover in $M_{\aleph_{\alpha}}$, $DC_{<\aleph_{\alpha}}$ holds since $\mathcal{I}$ is closed under $<\aleph_{\alpha}$ unions.  

\textbf{Remark 2.} In the permutation model $\mathcal{N}$ of \cite{Tac2016}, CS, as well as CWF, holds following the work in this section. Moreover, the following statement holds in $\mathcal{N}$, following the work in this section.

\begin{center}
    {\em `If in a partially ordered set, all chains are finite and all antichains are countable, then the set is countable.'}
\end{center}

\begin{thm}
{\em There is a permutation model $\mathcal{N}$ of $ZFA$, where there is an amorphous set.  Moreover, the following holds in $\mathcal{N}$.
\begin{enumerate}
    \item If in a partially ordered set, all chains are finite and all antichains are countable, then the set is countable.
    \item If in a partially ordered set, all chains are finite and all antichains have size $\aleph_{\alpha}$, then the set has size $\aleph_{\alpha}$ for any regular $\aleph_{\alpha}$.
\end{enumerate}}
\end{thm}

\begin{proof}
We consider the basic Fraenkel model (labeled as Model $\mathcal{N}_{1}$ in \cite{HR1998}) where 
`there are no amorphous sets' is false and $UT(WO,WO,WO)$ holds (c.f. \cite{HR1998}). 
Let $(P,\leq)$ be a p.o.set in $\mathcal{N}_{1}$, and $E$ be a ﬁnite support of $(P,\leq)$. By (1) in the proof of \textbf{Lemma 4.2}, $\mathcal{O}=\{Orb_{E}(p): p\in P\}$ is a well-ordered partition of $P$. Now for each $p\in P$, $Orb_{E}(p)$ is an antichain (c.f. the proof of \textbf{Lemma 9.3} in \cite{Jec1973}). Thus, by methods of \textbf{Lemma 4.2}, (1) and (2) hold in $\mathcal{N}_{1}$.
\end{proof}

\begin{thm}
{\em There is a permutation model of $ZFA$ where CS, as well as CWF, holds, but $AC^{\omega}_{fin}$ fails. Moreover, the following statements hold in the model.
    \begin{enumerate}
        \item For every $3\leq k<\omega$, $\mathcal{P}_{k}$ holds for any graph $G_{1}$ on some linearly-orderable set of vertices.
        \item For an infinite  graph $G=(V_{G}, E_{G})$ on a linearly-ordered set of vertices $V_{G}$ and a finite graph $H=(V_{H}, E_{H})$, if every finite subgraph of $G$ has a homomorphism into $H$, then so has $G$.
        
        \item For every finite field $\mathcal{F}=\langle F,...\rangle$, for every nontrivial linearly-orderable vector space $V$ over $\mathcal{F}$, there exists a non-zero linear functional $f:V\rightarrow F$.
    \end{enumerate}
}
\end{thm}

\begin{proof}
We recall the L\'{e}vy's permutation model (labeled as Model $\mathcal{N}_{6}$ in \cite{HR1998}).

\begin{itemize}
    \item \textbf{Defining the ground model $M$.} We start with a ground model $M$ of $ZFA+AC$ where $A$ is a countably infinite set of atoms written as a disjoint union $\cup\{P_{n}:n\in \omega\}$, where $P_{n}=\{a_{1}^{n},...a_{p_{n}}^{n}\}$ such that $p_{n}$ is the $n^{th}$-prime number. 
    \item \textbf{Defining the group $\mathcal{G}$  of permutations and the filter $\mathcal{F}$ of subgroups of $\mathcal{G}$.}
    \begin{itemize}
    \item \textbf{Defining $\mathcal{G}$.} $\mathcal{G}$ be the group generated by the following permutations $\pi_{n}$ of $A$. 
    
    \begin{center}
        $\pi_{n}: a_{1}^{n} \mapsto a_{2}^{n} \mapsto ... a_{p_{n}}^{n}\mapsto a_{1}^{n}$ and $\pi_{n}(x)=x$ for all $x\in A\backslash P_{n}$.
    \end{center}
    
    \item \textbf{Defining $\mathcal{F}$.} $\mathcal{F}$ be the filter of subgroups of $\mathcal{G}$ generated by $\{$fix$_{\mathcal{G}}(E): E\in [A]^{<\omega}\}$. 
    \end{itemize}
    \item \textbf{Defining the permutation model.} Consider the permutation model $\mathcal{N}_{6}$ determined by $M$, $\mathcal{G}$ and $\mathcal{F}$.
\end{itemize}

It is well-known that in $\mathcal{N}_{6}$, $AC^{\omega}_{fin}$ fails since $\{P_{i}:i\in \omega\}$ has no (partial) choice function (c.f. \cite{Jec1973}). Consequently, following \textbf{Lemma 4.4} of \cite{Tac2019a}, CAC fails in $\mathcal{N}_{6}$.
Since \textbf{every permutation $\phi\in\mathcal{G}$ moves only finitely many atoms}, following the arguments in \textbf{Lemma 4.2}, we can observe that CS, as well as CWF, holds in $\mathcal{N}_{6}$.

\begin{lem}
{\em In $\mathcal{N}_{6}$, LW holds.}
\end{lem}

\begin{proof}
Let $(X,\leq)$ be a linearly ordered set in $\mathcal{N}_{6}$ supported by E. We show fix$_{\mathcal{G}}E\subseteq$ fix$_{\mathcal{G}}X$ which implies that $X$ is well-orderable in $\mathcal{N}_{6}$. For the sake of contrary assume fix$_{\mathcal{G}}E\not\subseteq$ fix$_{\mathcal{G}}X$. So there is an element $y\in X$ which is not supported by $E$ and there is a $\phi\in$ fix$_{\mathcal{G}}E$ such that $\phi(y)\not=y$. Since $\phi(y)\not=y$ and $\leq$ is a linear order on $X$, we obtain either $\phi(y)<y$ or $y<\phi(y)$. Let $\phi(y)<y$. Since \textbf{every permutation $\phi\in\mathcal{G}$ moves only finitely many atoms} there exists some $k<\omega$ such that $\phi^{k}=1_{A}$. Thus, $p=\phi^{k}(p)<\phi^{k-1}(p)<...<\phi(p)<p$ which is a contradiction. Similarly we can arrive at a contradiction if we assume $y<\phi(y)$.
\end{proof}

Since LW holds in $\mathcal{N}_{6}$, we can observe (1), (2) and (3) in $\mathcal{N}_{6}$ by observations in \textbf{section 3}.
\end{proof}

\begin{thm}
{\em There is a permutation model of $ZFA$ where CS, as well as CWF, holds, but $LOKW_{4}^{-}$ fails. Moreover, the following statements hold in the model.
\begin{enumerate}
        \item For every $3\leq k<\omega$, $\mathcal{P}_{k}$ holds for any graph $G_{1}$ on some linearly-orderable set of vertices.
        \item For an infinite  graph $G=(V_{G}, E_{G})$ on a linearly-ordered set of vertices $V_{G}$ and a finite graph $H=(V_{H}, E_{H})$, if every finite subgraph of $G$ has a homomorphism into $H$, then so has $G$.
        
        \item For every finite field $\mathcal{F}=\langle F,...\rangle$, for every nontrivial linearly-orderable vector space $V$ over $\mathcal{F}$, there exists a non-zero linear functional $f:V\rightarrow F$.
    \end{enumerate}
}
\end{thm}

\begin{proof}
We recall the permutation model $\mathcal{M}$ from the second assertion of \textbf{Theorem 10(ii) of \cite{HT2019}}. 

\begin{itemize}
\item \textbf{Defining the ground model $M$.} Let $\kappa$ be any infinite well-ordered cardinal number. We start with a ground model $M$ of $ZFA+AC$ where $A$ is a
$\kappa$-sized set of atoms written as a disjoint union $\cup\{A_{\alpha}:\alpha< \kappa\}$, where $A_{\alpha}=\{a_{\alpha,1},a_{\alpha,2},a_{\alpha,3},a_{\alpha,4}\}$ such that $\vert A_{\alpha}\vert=4$ for all $\alpha<\kappa$.
    \item \textbf{Defining the group $\mathcal{G}$  of permutations and the filter $\mathcal{F}$ of subgroups of $\mathcal{G}$.}
    \begin{itemize}
    \item \textbf{Defining $\mathcal{G}$.}
    Let $\mathcal{G}$ be the weak direct product of $\mathcal{G}_{\alpha}$'s where $\mathcal{G}_{\alpha}$ is the alternating group on $\mathcal{A}_{\alpha}$ for each $\alpha<\kappa$.
    
    \item \textbf{Defining $\mathcal{F}$.} Let $\mathcal{F}$ be the normal filter of subgroups of $\mathcal{G}$ generated by $\{$fix$_{\mathcal{G}}(E): E\in [A]^{<\omega}\}$. 
    \end{itemize}
    \item \textbf{Defining the permutation model.} Consider the permutation model $\mathcal{M}$ determined by $M$, $\mathcal{G}$ and $\mathcal{F}$.
\end{itemize}

In $\mathcal{M}$, $LOKW_{4}^{-}$ fails (c.f. \textbf{Theorem 10(ii) of \cite{HT2019}}).
Since \textbf{every permutation, $\phi\in\mathcal{G}$ moves only finitely many atoms}, following the arguments in \textbf{Lemma 4.2} we can observe that CS, as well as CWF, holds in $\mathcal{M}$. 
Since LW holds in $\mathcal{M}$ (c.f. \textbf{Theorem 10(ii) of \cite{HT2019}}), we can observe (1), (2) and (3) in $\mathcal{M}$ by observations in \textbf{section 3}.
\end{proof}
\section{Observations in Howard's model}
\begin{thm} 
{\em For any $3\leq k<\omega$, $\mathcal{P}_{k}$ follows from F$_{k-1}$ in ZF. Moreover, if $X\in \{AC_{3}, AC^{\omega}_{fin}\}$, then the statement $\mathcal{P}_{k}$ does not imply $X$ in ZFA when $k=3$.}
\end{thm}

\begin{proof}
Fix $3\leq k<\omega$. Suppose $\chi(E_{G_{1}})=k$, $\chi(E_{G_{2}})\geq\omega$ and $G_{1}$ is a graph on some well-orderable set of vertices. First we observe that if $g:V_{G_{1}}\rightarrow \{1,...,k\}$ is a good $k$-coloring of $G_{1}$, then $G(\langle x,y\rangle)=g(x)$ is a good $k$-coloring of $G_{1}\times G_{2}$. So, $\chi(E_{G_{1}\times G_{2}})\leq k$. For the sake of contradiction assume that $F:V_{G_{1}}\times V_{G_{2}}\rightarrow \{1,...,k-1\}$ is a good coloring of $G_{1}\times G_{2}$. 
For each color $c\in \{1,...,k-1\}$ and each vertex $x\in V_{G_{1}}$ we let $A_{x,c}=\{y\in V_{G_{2}}:F(x,y)=c\}$. 
Define a relation $R$ on $\{1,...,k-1\}$ as $(v_{1},i)R(v_{2}, j)$ if and only if {\em `$v_{1}\not=v_{2}$ implies $A_{v_{1},i}\cap A_{v_{2},j}$ is not independent'} for $v_{1}, v_{2}\in V_{G_{1}}$.
By $F_{k-1}$ and \textbf{claim 3.8} there exist a choice function $f$ such that for any $x,x'\in V_{G_{1}}$, $A_{x,f(x)}\cap A_{x',f(x')}$ is not independent. 
Since $x\rightarrow f(x)$ is not a good coloring in $G_{1}$ as $\chi(E_{G_{1}})=k$, there are $x,x'\in V_{G_{1}}$ with $f(x)=f(x')=j$ and $\{x,x'\}\in E_{G_{1}}$. Consequently, $A'=A_{x,f(x)}\cap A_{x',f(x')}$ is not independent. Pick $y,y'\in A'$ joined by an edge in $E_{G_{2}}$. Then $(x,y)$ and $(x',y')$ are joined in $E_{G_{1}}\times E_{G_{2}}$ and get the same color $j$ which is a contradiction to the fact that $F$ is a good coloring of $G_{1}\times G_{2}$. 

For the second assertion, we consider the permutation model $\mathcal{N}$ from \textbf{section 3} of \cite{How1984} where $AC_{3}$ fails, and $F_{2}$ holds. Consequently, $\mathcal{P}_{3}$ holds in $\mathcal{N}$. In $\mathcal{N}$, there is a countable family $\mathcal{A}=\{A_{i}:i\in \omega\}$ which has no partial choice function. Consequently, $PAC^{\omega}_{fin}$ fails. Since $PAC^{\omega}_{fin}$ is equivalent to $AC^{\omega}_{fin}$ (see the proof of  \textbf{Lemma 4.4} of \cite{Tac2019a}), $AC^{\omega}_{fin}$ fails in $\mathcal{N}$.  
\end{proof}

\begin{question}
{\em If $k>3$, does UL follow from $\mathcal{P}_{k}$? Otherwise is there any model of ZF or ZFA, where $\mathcal{P}_{k}$ holds for $k>3$, but UL fails?}
\end{question}

\begin{thm}
{\em For any $2\leq k<\omega$, $\mathcal{P}_{G,H}$ restricted to finite graph $H$ with $k$ vertices follows from F$_{k}$ in ZF. Moreover, if $X\in \{AC_{3}, AC^{\omega}_{fin}\}$, then $\mathcal{P}_{G,H}$ restricted to finite graph $H$ with 2 vertices does not imply $X$ in ZFA.}
\end{thm}
\begin{proof}
Fix $2\leq k<\omega$. Let $V_{H}=\{v_{1},.... v_{k}\}$. For each $x\in V_{G}$, let $A_{x}=\{(x,v_{1}),...(x,v_{k})\}$. Define a relation $R$ on $\cup_{x\in V_{G}} A_{x}$ by $(x, v_{i})R(x', v_{j})$ if and only if {\em `$\{x,x'\}\in E_{G}$ implies $\{v_{i},v_{j}\}\in E_{H}$'}  for $(x,v_{i})\in A_{x}, (x', v_{j})\in A_{x'}$. 
By assumption, for all finite $F\subset V_{G}$, there exists a homomorphism $h_{F}:G\restriction F\rightarrow H$. For any finite $F\subset V_{G}$, and an homomorphism $h_{F}$ of $F$, let ${h_{F}}^{*}(j)=(j,h_{F}(j))$ for $j\in F$. Clearly, ${h_{F}}^{*}$ is an $R$-consistent choice function for $\{A_{x}\}_{x\in F}$. By $F_{k}$, there is a $R$-consistent choice function ${h_{F}}^{*}$ for $\{A_{x}\}_{x\in V_{G}}$. Define $h_{V_{G}}$ on $V_{G}$ by ${h_{V_{G}}}^{*}(j)=(j,h_{V_{G}}(j))$ for $j\in V_{G}$. Let $(j,j')\in E_{G}$ such that $j,j'\in V_{G}$. Since ${i_{I}}^{*}$ is $R$-consistent, $(j,h_{V_{G}}(j))R(j',h_{V_{G}}(j'))$. By the definition of $R$, $(h_{V_{G}}(j),h_{V_{G}}(j'))\in E_{H}$.

For the second assertion, we once more consider the permutation model $\mathcal{N}$ from \textbf{section 3} of \cite{How1984} where $AC_{3}$ and $AC^{\omega}_{fin}$ fails, and $F_{2}$ holds. Consequently, `$\mathcal{P}_{G,H}$ for a finite graph $H$ with 2 vertices' holds in $\mathcal{N}$. 
\end{proof}

\textbf{Acknowledgement.} We would like to thank the reviewer for reading the manuscript carefully and providing suggestions for improvement.

\end{document}